\documentclass[a4paper,11pt]{article} 
\usepackage{amssymb, latexsym, mathtools, amsmath, amsthm}
\usepackage{amsfonts}
\usepackage{color}
\usepackage{csquotes}
\usepackage{hyperref} 
\usepackage{cleveref}
\usepackage[shortlabels]{enumitem}
 
\usepackage{latexsym, amsfonts, amsthm, amsmath, amssymb, tikz, enumitem, booktabs, float,color,multicol,epsf, url,epsfig,epstopdf, array, setspace, graphicx, csquotes, xcolor,tikz-cd}

\usepackage{amssymb}
\usepackage{amssymb}
\usepackage{titlesec}
\usepackage[Rejne]{fncychap}
\usepackage{lipsum}
\usepackage{float}
\usepackage{xfrac} 
\usepackage[Rejne]{fncychap}
\usepackage{faktor}
\usepackage{nameref}

\newtheorem{theorem}{Theorem}[section]

\newtheorem{lemma}{Lemma}[section]
\newtheorem{corollary}{Corollary}[section]
\newtheorem{remark}{Remark}[section]
\newtheorem{proposition}{Proposition}[section]
\theoremstyle{definition}
\newtheorem*{algorithm}{Algorithm}

\def\r{{\sf{R}}}

\newcommand{\Z}{\mathbb Z}
\newcommand{\Q}{\mathbb Q}
\newcommand{\C}{\mathcal C}
\newcommand{\R}{\mathbb R}
\newcommand{\K}{\mathbb K}

\newcommand{\im}{\operatorname{im}}
\renewcommand{\span}{\operatorname{span}}

\newcommand{\myref}[1]{\hyperref[#1]{#1}}

\def \GL{\sf{GL}}

\newcommand{\SL}{\sf{SL}}
\newcommand{\M}{\sf{M}}
\newcommand{\PSL}{\sf{PSL}}
\newcommand{\PGL}{\sf{PGL}}
\newcommand{\Sp}{\sf{Sp}}
\newcommand{\Mat}{\sf{Mat}}

\begin{document}
\pagenumbering{arabic}
 \date{}
\author{\small Adem Zeghib}

\title{{\small\textbf{A CRITERION FOR TITS ALTERNATIVE ON THE CENTRALIZER OF A MATRIX}}}
\maketitle

\begin{abstract}
We give a necessary and sufficient condition on a matrix for its  centralizer in $\GL(n,\Z)$ to be  polycyclic,  or equivalently in this case,  not to contain a non-abelian free subgroup. We give a simple condition on the matrix ensuring that it is abelian. This can be thought of as an effective Tits alternative on centralizers in  $\GL(n,\mathbb{Z})$.   We apply these criteria to the conjugacy problem in certain arithmetic groups preserving a non-degenerate $\Q$-bilinear form, such as integral symplectic groups. We derive an effective solution to the conjugacy problem in such groups when given matrices satisfy the above criterion. This solution is based on effective solutions to the conjugacy problem in $\GL(n,\Z)$ by Eick-Hofmann-O’Brien and to an orbit problem for polycyclic groups, by  Eick and Ostheimer.
\end{abstract}


\section{Introduction}

Let $\r$ be a ring and $T \in \GL(n, \mathbb{\r})$ a matrix. The centralizer of $T$ in  $\GL(n, \r)$,

$$\mathcal{C}_{\GL(n, \r)}(T) = \{ A \in \GL(n, \r) \mid \text{$AT=TA$} \}$$

is the subgroup of all invertible matrices that commute with $T$.
When it is algorithmically relevant, the \emph{centralizer problem} for $T$ is the problem of giving an explicit generating set for $\C_{\GL(n,\r)} (T)$.  A dual problem is the \emph{conjugacy problem} in $\GL(n,\r)$: given two matrices  $T, \widehat{T} \in \GL(n, \r)$, determine whether there exists $P \in \GL(n, \r)$ such that $\widehat{T} = PTP^{-1}$. 

In \cite{eick2019conjugacy}, Eick, Hofmann and O’Brien  developed practical algorithms for the conjugacy and centralizer problems in $\GL(n, \Z)$ building on an abstract solution to the latter problems due to Grunewald in \cite{grunewald1980solution}. Subsequently, Bley, Hofmann, and Johnston improved the conjugacy problem aspect of this work in \cite{bley2022computation}, without addressing the centralizer problem. However, the structural nature of these centralizers is not evidently revealed. For instance, it is natural to ask how the Tits alternative \cite{tits1972free} applies to these subgroups, i.e. when does $\C_{\GL(n,\Z)}(T)$ contain a non abelian free subgroup and when is it abelian or virtually solvable.

In this paper, we focus on the case where $\r$ is a field $\mathbb{K}$ or the ring of integers $\mathbb{Z}$. In the first case,  if $\K$ is algebraically closed, the answer, grounded in linear algebra, is rather pleasant, although hard to find in the literature (all references we found treat only special cases).  We provide a full concise account of it in Theorem \ref{theorem3.1} for $\overline{\K}$ algebraically closed : if the Jordan
reduction of $T$ has at least two blocks of the same size for the same eigenvalue, then its centralizer contains a copy of $\GL(2,\K)$ ; otherwise, it is solvable. It is even abelian
if and only if $T$ is non-derogatory, i.e. the characteristic and minimal polynomials of $T$
coincide.

When $T$ is in $\GL(n,\Q)$, the intersection of its centralizer with $\GL(n,\Z)$ does not appear directly related to the algebraic Jordan decomposition. Leveraging deep theorems from the theory of arithmetic groups, we are able to extend the result and establish the following trichotomy.

\begin{theorem}\label{theorem1}
Let $T \in \GL(n, \Q)$ and let $\C_{\Z}:=\C_{\GL(n,\Z)}(T)$.
\begin{itemize}
\item[\rm (a)]
If $T$ has a unique Jordan block for every eigenvalue, then $\C_{\Z}$ is abelian.
\item[\rm (b)]
If $T$ has two Jordan blocks of the same size for the same eigenvalue, then $\C_{\Z}$ contains a free group on two generators.
\item[\rm (c)]
If each eigenvalue has Jordan blocks only of different size, and at least one
eigenvalue has several Jordan blocks, then $\C_{\Z}$ is polycyclic.
\end{itemize}
\end{theorem}

Establishing a criterion for the polycyclicity of centralizers is interesting considering the vast literature on this class of group, which in particular allowed the development of software tools in GAP and related packages.
Segal provided a comprehensive overview of polycyclic groups in \cite{segal1983polycyclic}. Eick, Assmann, and Ostheimer addressed computational aspects of polycyclic matrix groups, including presentations and more general algorithms, as detailed in \cite{MR1716421} ,\cite{Assmann} and \cite{eick2022polycyclic}.

We now explain how Theorem \ref{theorem1} can be applied in the context of the conjugacy problem for certain arithmetic groups. We first recall the context. The conjugacy problem and, more precisely, the description of conjugacy classes in unitary, symplectic and orthogonal groups over a field has been extensively studied by Springer \cite{springer1951symp} and Wall \cite{wall1963conj}. 
In the case of integer coefficients, consider a subgroup $H$ of $\GL(n,\Z)$  that preserves a non-degenerate $\Q$-bilinear form (for instance, O$(p, n-p)(\Z)$, or $\Sp(n,\Z)$ if $n$ is even). If $T$ and $\widehat{T}$ are two matrices in $H$, the conjugacy problem is solved in principle using theoretical methods developed by Grunewald and Segal \cite{grunewald1980some}, but there is no hope of implementing this. In contrast, Eick, Hofmann, and O'Brien have devised another algorithm and accomplished an actual implementation of the latter algorithm,  solving the conjugacy problem in $\GL(n,\Z)$, \cite{eick2019conjugacy}. Using this algorithm directly in $H$ provides only partial information (namely, whether the two matrices are conjugate by a matrix a priori in $\GL(n,\Z)$, not necessarily in $H$).  We would then need to ensure, given a conjugator in $\GL(n,\Z)$, whether there is another conjugator in $H$. This can be expressed as an orbit problem for the centralizer of $T$, acting on the space of bilinear forms (see Proposition \ref{proposition15}). An algorithm implemented by Eick and Ostheimer \cite{eick2003orbit} solves this problem when the acting group is polycyclic. Putting together \cite{eick2019conjugacy}, \cite{eick2003orbit}, Proposition \ref{proposition15} and our criterion for the Tits alternative, Theorem \ref{theorem1}, it allows us to approach the conjugacy problem in $H$ for matrices with a polycyclic centralizer. This is addressed in Section \ref{section3.2}, leading to the following corollary: 

\begin{corollary}
    Let $b$ be a $\Q$ non degenerate bilinear form, $M$ its matrix in the canonical basis and $H_M$ the subgroup of $\GL(n,\Z)$ preserving the bilinear form $b$.
    The problem of whether two given matrices in $\GL(n,\Q)$ are conjugate by an element of $H_M$ is decidable. More precisely, Algorithm \myref{A} decides whether these matrices are conjugate in $H_M$ and, provides a conjugating element if it exists.
\end{corollary}

\section*{\small Aknowledgements}
I would like to sincerely thank my supervisor Bettina Eick for her guidance, valuable advice, and patience in explaining the orbit stabilizer algorithm and the orbit problem related to the conjugacy problem. I would also like to express my sincere gratitude to my supervisor, François Dahmani, for his patience and support during this project.

\section{Trichotomy on the centralizer of a matrix}

\subsection{Centralizers over an algebraically closed field}

In the following $\mathbb{K}$ is a field, and
$\overline{\mathbb{K}}$ its algebraic closure. We start by recalling the classical Jordan normal form theorem. 

\begin{theorem}[Jordan normal form]
Let $T \in \GL(n,\K)$ then there exists a unique set of matrices $\{J_1,...,    J_k\} $  with $J_i =\begin{pmatrix}
\lambda_i & 1& 0 & 0 \\
0 & \lambda_i & \ddots & 0\\
0 & 0 & \ddots & 1  \\
 0&  0  &  0& \lambda_i
\end{pmatrix} \in \GL(r_i,\overline{\K})$, a unique tuple $(m_1,m_2,...,m_k)$ and $V \in \GL(n,\overline{\K})$ (not necessarily unique) that satisfy 
$$ VTV^{-1}=diag(J_1,J_1,...,J_2,...,J_k) \quad \text{where $J_i$ appears $r_i$ times and $\sum r_{i}\times{m_i}=n$} .$$

\end{theorem}

This has the following consequence on centralizers in $ \GL(n,
 \overline{\mathbb{K}})$.

\begin{theorem}\label{theorem3.1}
\label{derog}
Let $T \in \GL(n, \K)$ and let $\C_{\overline{\K}}=\C_{\GL(n,\overline{\K})}(T)$ its centralizer with coefficient in $\overline{\K}$.

\begin{itemize}
\item[\rm (a)]
If $T$ has a unique Jordan block for every eigenvalue, then $\C_{\overline{\K}}$ is abelian.
\item[\rm (b)]
If $T$ has two Jordan blocks of the same size for the same eigenvalue, then $\C_{\overline{\K}}$ contains a copy $\GL(2,\overline{\K})$.
\item[\rm (c)]
If  each eigenvalue has several ($\geq 1$) Jordan blocks, but only of different sizes, and one eigenvalue has at least two blocks,  then $\C_{\overline{\K}}$ is a non abelian solvable group.
\end{itemize}
\end{theorem}

Observe that the three assumptions are mutually exclusive and cover all cases.

\begin{remark}\label{remark4}
Given $T\in \GL(n,\Q)$, it is possible to effectively determine in which case $T$ falls, proceeding as follows. First, compute its characteristic polynomial $\chi$, and its minimal polynomial  $\mu$.    If $\chi = \mu$, we are in case (a). Assume that they differ. Compute the $\mathbb{Z}$-irreducible factorization of $\chi$. Then check whether there is an irreducible factor $\pi$ of $\chi$  and $j \geq 1$  for which     \[\left( dim(\ker(\pi^j(T)) - dim(\ker(\pi^{j-1}(T)\right) - \left( dim(\ker(\pi^{j+1}(T)) - dim(\ker(\pi^{j}(T)) \right) >deg(\pi) .  \] If there exists such an integer,  we are in case (b). If for any factor $\pi$ there is no such integer, the iterated kernels dimension grows only by $0$ or $1$ for every root of $\pi$  at any step, and  we are in case (c).
\end{remark}

First, let us prove claim (a) (which is the case of non-derogatory matrices). The centralizer of a single Jordan matrix is abelian, a property that follows from the classical fact that such matrices admit a cyclic vector. In general, the centralizer of an endomorphism is also the direct product of the centralizers of  its restrictions to generalized eigenspaces. Here, those spaces correspond to each Jordan block, the centralizer is a direct product of abelian subgroups. \newline

We now prove the point $(b)$.

\begin{proposition}\label{prop4}
If $T \in \GL(n,\K)$ has two Jordan blocks of the same size for the same eigenvalue, then $\C_{\GL(n,\overline{\K})}(T)$ contains a copy of $\GL(2,\overline{\K})$. Moreover, if $\K=\R$, $\C_{\GL(n,\R)}(T)$ contains a copy of $\GL(2, \R)$.
\end{proposition}

\begin{proof}
The matrix $T$ is conjugate in $\GL(n,\overline{K})$ to the following matrix
    $T \sim_{\GL(n,\overline{\K})}
J'=\begin{pmatrix}
J & 0& 0 \\
0 & J & 0\\
0 & 0 & A 
\end{pmatrix} $
then 
$\begin{pmatrix}
aI & bI & 0 \\
cI & dI & 0\\
0 & 0 & I 
\end{pmatrix} $
commutes with $J'$ for every $a,b,c,d \in \overline{\K} $ and is invertible as long as $ad-bc\ne 0$.
The following map is an injective homomorphism : 

   \[ \Psi  :
\begin{array}{ccccc}
&  \GL(2,\overline{\K})  & \longrightarrow & \GL(n,\overline{\K}) & \\
& \begin{pmatrix}
a & b  \\
c & d 
\end{pmatrix}  & \longmapsto &
\begin{pmatrix}
aI & bI & 0 \\
cI & dI & 0 \\
0 & 0 & I
\end{pmatrix} \\
\end{array}
 \]

If $\K=\R$, we need the following lemma. 

\begin{lemma}\label{lemma5}
    Let $T$ be a real matrix such that $T$ is $\mathbb{C}$-conjugate to $diag(J(\lambda),A )$, with $J(\lambda)$ a Jordan block of size $s$. \newline
    If $\lambda \notin \R$ then
    \begin{itemize}
        \item $T$ is $\mathbb{C}$-conjugate to $diag(J(\lambda), J(\overline{\lambda})), A')$, with $J(\overline{\lambda})$ also of  size $s$.
        \item $T$ is $\R$-conjugate to $diag(N,B)$ with $N$ and $B$ real matrices.
    \end{itemize}
\begin{proof}
    First, since $T$ is a real matrix, then $\overline{T}=T$, which leads to the existence of a Jordan block $J(\overline{\lambda})$ of the same size.
    Then, we can easily check that 
$$\begin{pmatrix}
I_s & I_s \\
iI_s & -iI_s  
\end{pmatrix}^{-1}= \frac{1}{2}\begin{pmatrix}
I_s & -iI_s \\
I_s & iI_s  
\end{pmatrix}$$ 

and that the conjugate

$$\begin{pmatrix}
I_s & I_s \\
iI_s & -iI_s  
\end{pmatrix}
\begin{pmatrix}
J(\lambda) & 0 \\
0 & J(\overline{\lambda})  
\end{pmatrix}
\begin{pmatrix}
I_s & -iI_s \\
I_s & iI_s  
\end{pmatrix} \quad \text{is a real matrix}$$

We can continue in the same fashion and gather Jordan blocks of the same size that are conjugate to obtain the matrix $B$. Therefore, $T$ and $diag(N,B)$ are $\mathbb{C}$-conjugate. We conclude using the classical fact that if two real matrices are $\mathbb{C}$-conjugate they are also $\R$-conjugate.

\end{proof}
    
\end{lemma}

We continue our proof. Let $T$ be a real matrix as in the statement of Proposition \ref{prop4}.  Using the last lemma we get that $T$ is $\R$-conjugate to $diag(N,N,A')$ with $N$ and $A$ real matrices. We conclude using the homomorphism $\Psi$ as before to show that there is a copy of $\GL(2,\R)$ in $\C_{\GL(n,\R)}(T)$.

\end{proof}                                   

 We now want to prove the claim $(c)$. As argued earlier (for case (a)), the centralizer of $T$ is a direct sum of the centralizers of the restrictions to generalized eigenspaces. We can therefore reduce our study to the case of matrices with a single eigenvalue.

\begin{proposition}\label{prop6}
    Let $T$ be a matrix with one eigenvalue. If all Jordan blocks of $T$ are of different sizes, then there exists a flag  of $\K^n$ that is preserved by its centralizer.
    
\end{proposition}

\begin{proof}
Considering $T-\lambda I$, we can suppose that $T$ is a nilpotent matrix and its centralizer is the same as $T$.
Take $e_1,e_2,...,e_n$ a basis for a Jordan normal form.
Now for each $e_i$, we set $k_i :=\text{min}\{j: e_i \in \ker(T^{j})\} $ and $r_i=\text{max} \{j:e_i \in \im(T^{j})\}$. 
Since all Jordan blocks have different sizes, one can check that $e_i$ is the $k_i^{th}$ vector of its Jordan block that has size $k_i+r_i$.
Therefore, the couple $(k_i,-r_i)$  uniquely determines the vector $e_i$ among the set $\{e_j, \quad 1 \le j\le  n\}$.
We take the lexicographic order on the set $(k_i,-r_i) \in \Z^2$ and arrange the vectors in increasing order to obtain the basis $\mathcal{F}:=(f_1,f_2,...,f_n)$. Let us show that $\span (f_1,...,f_i)=\ker(T^{k_i})\cap \im(T^{r_i})$. \newline
By construction, we have that $\span(f_1,...,f_i) \subset \ker(T^{k_i})\cap \im(T^{r_i})$.
We need to show that $\ker(T^{k_i}) \cap \im(T^{r_i}) \subset \span (f_1,...,f_i)$.
By construction of the basis $\mathcal{F}$, $\ker(T^{k_i})= \span (f_1,...,f_i,f_{i+1},...,f_{j})$ such that $f_i,f_{i+1},...f_j$ have the same $k_i$ and $f_{j+1} \notin \ker(T^{k_i})$. For all $k\in \{i+1,...,j\}$, $r_i> r_k$, which leads to $\ker(T^{k_i}) \cap \im(T^{r_i}) \subset \span (f_1,...,f_i)$.
The centralizer preserves $\ker(T^{k_i})$ and $\im(T^{r_i})$, so it also preserves the flag $\span(f_1,...,f_i)=\ker(T^{k_i}) \cap \im(T^{r_i})$.

\end{proof}

\begin{corollary}\label{cor1}
Let $T\in \GL(n,\K)$ be a matrix with one eigenvalue with all its Jordan blocks of different sizes. Then $\C_{\GL(n,\overline{\K})}(T)$ is solvable.
\end{corollary}

\begin{proof}
Let $A$ be a matrix that commutes with $T$. Proposition \ref{prop6} provides a total flag that must be preserved by all matrices that commute with $T$. Therefore, there is a common trigonalisation basis for all elements of $\C_{\GL(n,\overline{\K})}(T)$. Since the group of upper triangular matrices is solvable of length n, and all subgroups of solvable groups are solvable, the group $\C_{\GL(n,\overline{\K})}(T) $ is solvable of length at most $n$.
\end{proof}

\begin{remark}
One should be cautious that the total flag (and therefore the trigonalisation basis) proposed by Proposition \ref{prop6} is not the natural basis of the Jordan decomposition of a matrix. In particular, the matrices proposed in the next proof are not in contradiction with our discussion in Proposition \ref{prop6}. 
    
\end{remark}

\begin{proposition}\label{prop8}
    Let $T$ be a matrix with one eigenvalue and with at least two Jordan blocks. Then $\C_{\GL(n,\overline{\K})}(T)$ is non abelian.
\end{proposition}

\begin{proof}
Consider the matrix $J:=\begin{pmatrix}
J_r(\lambda) & 0& 0 \\
0 & J_k(\lambda) & 0\\
0 & 0 & J' 
\end{pmatrix}$ with $r \ge k$.

We define $I_{p, q}=\begin{pmatrix}
0_{p, (q-p)} & I_p \\
\end{pmatrix}$
if $p\le q$
and 
$I_{p, q}=\begin{pmatrix}
I_q\\ 0_{(p-q), q} \\
\end{pmatrix}$ if $p\ge q$.

Take $A:=\begin{pmatrix}
I_r & I_{r, k}& 0 \\
0 & I_k & 0\\
0 & 0 & I_l 
\end{pmatrix}$
and $B:=\begin{pmatrix}
I_r &0 & 0 \\
I_{k, r}& I_k & 0\\
0 & 0 & I_l 
\end{pmatrix}$.

We show that $A$ and $B$ belong to $\C_{\GL(n,\overline{\K})}(T)$ and do not commute.
A direct computation shows that $AJ=JA$, $BJ=JB$, and $det(A)=det(B)=1$. 
We have  $ABe_r=A(e_r+e_{r+k})=e_r+e_{r+k}+e_{k}$
and $BAe_r=Be_r=e_r+e_{r+k}$. This shows that $A$ and $B$ do not commute.

\end{proof}

Proof of $(c)$ : using Corollary \ref{cor1},  $\C_{\GL(n,\overline{\K})}(T)$ is a solvable group, and Proposition \ref{prop8} tells us that it is a non abelian group.

We have proven Theorem \ref{theorem3.1}.

\subsection[Centralizer over Z]{Centralizer over \(\Z\)}

 In this part, we will take $T\in \GL(n,\Q)$, this assumption makes $\C_{\GL(n,\Q)}(T)$ a $\Q$-algebraic group, i.e. a subgroup of $\GL(n,\mathbb{C})$ defined by polynomial equations with coefficients in $\Q$. This assumption is crucial for this part. We want to improve the trichotomy in the case of coefficients in $\Z$. In particular, we want to prove the following.
 
 \begin{proposition}\label{prop10}
Let $T\in \GL(n,\Q)$. If a non abelian free group is contained in $\C_{\GL(n,\mathbb{C})}(T)$ then the group $\C_{\GL(n,\Z)}(T)$ also contains a non abelian free group.
 \end{proposition}

We already know that if a non abelian free group
is contained in $\C_{\GL(n,\mathbb{C})}(T)$ then $\C_{\GL(n,\R)}(T)$ contains a copy of
$\GL(2,\R)$ by the trichotomy, and Proposition \ref{prop4}.

The next two propositions discuss a crucial and non trivial step in our proof. We show that if a copy of $\GL(2,\R)$ is contained in the Lie group $\C_{\GL(n,\R)}(T)$, then  its arithmetic part $\C_{\GL(n,\Z)}(T)$ contains a non abelian free subgroup.

\begin{proposition}\label{prop11}
 Let $T\in \GL(n,\R)$. If $\C_{\GL(n,\R)}(T)$ contains a copy
of $\GL(2,\R)$, then the quotient of $\C_{\GL(n,\R)}(T)$ by the maximal normal solvable subgroup (the solvable radical) is not compact. 
\end{proposition}

\begin{proof}
Recall that the solvable radical of a Lie group is closed (see Theorem 3.18.13 in \cite{varadarajan2013lie}), therefore, the quotient $\widehat{G}$ of $\C_{\GL(n,\R)}(T)$  by its solvable radical is a finite dimensional Lie group (see Theorem 2.9.6 in \cite{varadarajan2013lie}). Observe that the copy of $\GL(2,\R)$ of the assumption cannot intersect the radical solvable more than by its center (the only  normal subgroup of   $\GL(2,\R)$ that is solvable). Therefore, $\widehat{G}$ contains a copy of $\GL(2,\R)$ or $\PGL(2,\R)$. Note that $\PGL(2,\R)=\PSL(2,\R)$, so in both cases it contains a non-trivial image of $\SL(2,\R)$.
If the quotient of $\C_{\GL(n,\R)}(T)$ was compact,  this would provide a compact finite
dimensional representation of $\SL(2,\R)$.  However, using classical tools from Lie algebra, it is shown in Proposition 5.5 of \cite{zbMATH06064700}, that, except for the trivial one, there is no compact finite dimensional representation in $\SL(2,\R)$.

\end{proof}

The next proposition uses two non trivial features, one explains the behavior of arithmetic groups with respect to quotients, and one is the theorem of Borel Harish-Chandra \cite{borel1962arithmetic} that states that arithmetic groups in semi simple Lie groups over $\Q$ are lattices.  For this theorem, we use the formulation given in \cite{platonov1993algebraic}, since it is more suitable for our study.

\begin{theorem}[Theorem 4.14 \cite{platonov1993algebraic}]
Let $G$ be a semisimple $\Q$-group, and let $\Gamma \subset G_\R$ be an arithmetic subgroup. Then $\Gamma$ is a lattice in $G_\R$.
    
\end{theorem}

\begin{proposition}\label{prop13}
Let $T\in \GL(n,\Q)$. The quotient of $\C_{\GL(n,\R)} (T)$  by the solvable radical is not compact if and only if   
$\C_{\GL(n,\Z)} (T) $ contains a non abelian free subgroup.
\end{proposition}

\begin{proof}

The assumption means that the quotient by the solvable radical
is a non-compact semi-simple Lie group defined over the rational. By
Theorem 4.1 in \cite{platonov1993algebraic}, the image of the integer points of
$\C_{\GL(n,\R)} (T) $ (i.e. $ \C_{\GL(n,\Z)} (T))$ in the quotient is so-called arithmetic in it. By
Borel-Harish Chandra's Theorem \cite{borel1962arithmetic}, it is therefore a lattice in a non-compact semisimple Lie
group. Borel proved that such lattices are Zariski dense \cite{borel1960density}. In particular, this implies that they cannot be virtually solvable. By the Tits alternative, they must contain a non abelian free subgroup \cite{tits1972free}. Since no additional relations are imposed, this subgroup lifts straightforwardly to $\C_{\GL(n,\R)}(T)$. \\
Conversely, suppose that $\C_{\GL(n,\Z)}(T)$ contains a non abelian free subgroup, then $\C_{\GL(n,\mathbb{C})}(T)$ does as well. Consequently, according to Theorem \ref{theorem3.1}, the centralizer $\C_{\GL(n,\C)}(T)$ contains a copy of $\GL(2,\mathbb{C})$. Applying Proposition $\ref{prop4}$, we conclude that it actually contains a copy of $\GL(2,\R)$. Finally, Proposition \ref{prop11} implies that the quotient of $\C_{\GL(n,\R)}(T)$ by its solvable radical is not compact. 

\end{proof}

By Propositions \ref{prop11} and \ref{prop13} we have proven Proposition \ref{prop10}.

We give our final statement on the integral centralizer of a matrix in $ \GL(n,\Q
)$.
\begin{theorem}
Let $T \in \GL(n, \Q)$ and let $\C_{\Z}:=\C_{\GL(n,\Z)}(T)$.
\begin{itemize}
\item[\rm (a)]
If $T$ has a unique Jordan block for every eigenvalue, then $\C_{\Z}$ is abelian.
\item[\rm (b)]
If $T$ has two Jordan blocks of the same size for the same eigenvalue, then $\C_{\Z}$ contains a free group on two generators.
\item[\rm (c)]
If each eigenvalue has Jordan blocks only of different size, and at least one
eigenvalue has several Jordan blocks, then $\C_{\Z}$ is polycyclic.
\end{itemize}
\end{theorem}

\begin{proof}
  Since $\C_{\Z}$ is a subgroup of $\C_{\GL(n,\mathbb{C})}(T)$, we easily get $(a)$ from Theorem \ref{theorem3.1}. We also have, under the assumption of $(c)$, that $\C_{\Z}$ is solvable. It is also $\Z$-linear since it is in $\GL(n,\Z)$. Therefore, by Malcev \cite{mal1951certain} it is polycyclic and we have claim $(c)$.
    If $T \in \GL(n,\R)$ has two  blocks of the same size for the same eigenvalue. There is a copy of $\GL(2,\mathbb{C})$ in $\C_{\GL(n,\mathbb{C})}(T)$ and we showed in Proposition \ref{prop13} that in this case, a free group on two generators is actually contained in $\C_{\Z}$ which proves $(b)$. 
\end{proof}

\begin{remark}
    In contrast to Theorem \ref{theorem3.1}, it is not clear whether the three assumptions appearing in Theorem \ref{theorem1} are mutually exclusive. More precisely, it remains to identify whether there exists $T$ in $\GL(n,\Q)$  such that its centralizer $\C_{\GL(n,\R)}(T)$ is solvable but non abelian, while the arithmetic centralizer $\C_{\GL(n,\Z)}(T)$ is abelian.
\end{remark}

\begin{remark}

By Remark \ref{remark4}, one can effectively determine in which case the centralizer of a given matrix falls. 

\end{remark}

\section{Conjugacy problem and orbit stabilizer problem in groups preserving a form}\label{section3}

\subsection[Reduction of the conjugacy problem in HM to an orbit problem] {Reduction of the Conjugacy Problem in $H_M$ to an Orbit Problem}

Consider a non degenerate bilinear form $b$ on $\Z^{n}$ and its matrix $M:=(b(e_i,e_j))_{1 \le i,j \le n}$ in the canonical basis. The matrix $M$ belongs to $\M_n(\Z)$ and is invertible. The group of automorphisms of $\Z^n$ that preserve $b$ is canonically isomorphic to 

$$H_M := \{ A \in \GL(n,\Z) \mid \text{$AMA^t = M$}\}.$$

For example, if $b$ is a symplectic form, then $M$ is a skew symmetric matrix, and $H_M$ is an integral symplectic group. On the other hand, if $b$ is symmetric, then $M$ is a symmetric matrix and $H_M$ is an integral orthogonal group (its rank, in the sense of arithmetic groups, depends  on the signature of $b$). \\
The conjugacy problem in $H_M$ is a special case of the orbit problem
for the action of $H_M$ on $\GL(n,\Q)$ by conjugation, i.e. deciding whether given matrices $T$ and $\widehat{T}$ in $\GL(n,\Q)$ are conjugate by an element in $H_M$. We consider this problem, together with another orbit problem, as follows. Consider the action $\alpha_T$ defined by

    \[
\begin{array}{ccccc}
\alpha_T:
&  \C_{\GL(n,\Z)}(T) \times \M_{n}(\Q) &  \longrightarrow \M_{n}(\Q) & \\
& (C,A) & \longmapsto CAC^t.\\
\end{array}
 \]
 The associated orbit problem asks whether two given elements of $H_M$ are in the same $\alpha_T$-orbit.
In this section, we write $\C_{\Z}:=\C_{\GL(n,\Z)}(T)$ to shorten the notation.

\begin{proposition}\label{proposition15}
Let $T, \widehat{T} \in \GL(n, \Q)$. The following statements are equivalent:
\begin{itemize}
    \item[(i)] $T$ and $\widehat{T}$ are conjugate by an element of $H_M$.
    \item[(ii)] There exists $P_0 \in \GL(n,\Z) $  such that  $\widehat{T}=P_0TP_0^{-1}$ and for all such $P_0$, there exists $C_0 \in \C_\Z$ such that $C_0MC_0^t=P_0^{-1}M(P_0^{-1})^t$.
    \item [(iii)] There exists a matrix $P_0$ of $\GL(n,\Z)$ such that $\widehat{T}=P_0TP_0^{-1} $ and for all such $P_0$, the matrices $M$ and $P_0^{-1}M(P_0^{-1})^t$ are in the same $\alpha_T$-orbit.
     
\end{itemize}
\end{proposition}

\begin{proof}
Assume $(i)$ and let us show that it implies $(ii)$. We are thus given $P_0 \in H_M$ such that  $\widehat{T}=P_0TP_{0}^{-1}$. We have $M=P_{0}^{-1}M(P_{0}^{-1})^t$ because $P_{0}^{-1} \in H_M$. Thus, $(ii)$ holds for $C=I_n$.

Assume $(ii)$ and let us show that it implies $(i)$. Now $P_0 \in \GL(n,\Z)$ and $C_0 \in \C_{\Z}$ are given. It is a classical fact that the set $\{ P\in \GL(n,\Z) ;\quad \text{$\widehat{T}=PTP^{-1}$} \}$ is $P_0 \cdot \C_{\Z}$. We want $P_0 \cdot \C_{\Z}$ to intersect $H_M$. Consider $P_0C_0$: one has $P_0C_0 \in P_0\C_{\Z}$ and $(P_0C_0)M(P_0C_0)^t=M$. Thus, $P_0C_0 \in H_M$ and $(P_0C_0)T(P_0C_0)^{-1}=\widehat{T}$.

$(iii)$ is a reformulation of $(ii)$ using the action $\alpha_T$.

\end{proof}

\subsection[An algorithm for the conjugacy problem in HM for non derogatory matrices]{An algorithm for the conjugacy problem in $H_M$ for non derogatory matrices}\label{section3.2}
In order to decide, given $T, \widehat{T}$ as in Proposition \ref{proposition15},  whether
they are conjugate by an element of $H_M$, it remains to decide whether
there exists
$C_0 \in \C_{\Z}$ with $C_0MC_0^t = M'$  for an explicit matrix $M'$ as in the
third point of Proposition \ref{proposition15}. We can answer this question whenever the acting group is a polycyclic matrix group using \cite{eick2003orbit}. We restrict ourselves to a theoretical description of the algorithm and omit any implementation, as implementing such an algorithm is highly non trivial. This implementation is carried out by Timo Velten \cite{velten2026polycyclic} in the technically demanding setting of a matrix with a polycyclic centralizer. In the following, we outline the algorithm that solves the conjugacy and centralizer problem in $H_M$ under the assumption that $\C_{\GL(n,\Z)}(T)$ is polycyclic.
We get the following: 

\begin{corollary}
  Let $b$ be a $\Q$ non degenerate bilinear form, $M$ its matrix in the canonical basis and $H_M$ the subgroup of $\GL(n,\Z)$ preserving the bilinear form $b$.
    The problem of whether two given matrices in $\GL(n,\Q)$ are conjugate by an element of $H_M$ is decidable. More precisely, Algorithm \myref{A} decides whether these matrices are conjugate in $H_M$ and, provides a conjugating element if it exists.
\end{corollary}

Recall that we can detect whether $\mathcal{C}_{\GL(n,\Z)}$ is polycyclic or not using Remark \ref{remark4}.

In  \cite{eick2019conjugacy}, the authors present the following algorithms, which are implemented in MAGMA.
\begin{itemize}
    \item \textbf{Algorithm I} For $T, \widehat{T} \in \GL(n,\Q)$, say whether there exists $P_0\in \GL(n,\Z)$ such that $\widehat{T}=P_0TP_0^{-1}$, and determine $P_0$ if it exists.
    \item \textbf{Algorithm II} For $T \in \GL(n,\Q)$, determine a set of generators for $\C_{\GL(n,\Z)}(T)$.
    
\end{itemize}

In \cite{eick2003orbit},  the authors present the following algorithms, which are implemented in GAP. \newline
For $G$ a polycyclic group and $\langle g_1,...,g_k \rangle $ a presentation : 
\begin{itemize}
    \item \textbf{Algorithm III} For $v,w \in \Q^m$, say whether there exists $g \in G$ such that $gv =w$ and if so give $g$ is returned decomposed in the generating set $\langle g_1,...,g_k \rangle$.
    \item \textbf{Algorithm IV} For $v \in \Q^m$ compute a set of generators for elements $g\in G$ that satisfies $gv=v$.
\end{itemize}

For an element $C \in \GL(n,\Q)$, we consider the endomorphism $\alpha_C : X \longmapsto CXC^t$. We denote by $\Mat(\alpha_C) \in \GL(n^2,\Q)$ the matrix of the endomorphism $\alpha_C$ in the canonical basis of $\M_n(\Q)$.

To decide whether two matrices  $T,\widehat{T} \in \GL(n,\Q)$ are conjugate in $H_M$, with the condition that $\C_{\GL(n,\Z)}(T)$ is polycyclic, and if so determine a conjugating element, we propose the following algorithm :
\begin{algorithm}[\textbf{A}]\phantomsection\label{A}
Let $T, \widehat{T} \in \GL(n,\Q) $.
\begin{itemize}
    \item Decide if $T$ and $\widehat{T}$ are conjugate in $\GL(n, \Z)$ using \textbf{Algorithm I}; if they are return $P_0$ the conjugating matrix; if not, then return false.
    \item Using  \textbf{Algorithm II}, compute $C_1,...,C_k $ such that $\langle C_1,...,C_k\rangle= \C_{\GL(n,\Z)}(T)$.
    \item For every $1 \le i \le k$ construct the matrices $c_i:=\Mat(\alpha_{C_i})$. Set $G:= \langle c_1,...,c_k\rangle$.
    \item Transform the matrices according to the basis  $P_0$ and $P_0^{-1}M(P_0^{-1})^t$ in a vector of size $n^2$, respecting the canonical basis of $\M_n(\Q)$. Denote them by $v,w$.
    \item Using \textbf{Algorithm III}, decide if the vector $v$ and $w$ are in the same $G$-orbit. If they are then consider $c$ such that $cv=w$ and decompose $c$ in the generating set $\{c_1,...,c_k\}$. If not, return false.
    \item Using the decomposition of $c=c_{i_1}c_{i_2}...c_{i_j}$ in the generating set $\{c_1,...,c_k\}$, take  $C =C_{i_1}C_{i_2}...C_{i_j}$. 
    \item Return yes, and $P_0C$.

\end{itemize}

\end{algorithm}
The following algorithm computes a generating set for $\C_{H_M}(T)$ i.e. a generating set for $\C_{\GL(n,\Z)}(T)\cap H_M $.
\begin{algorithm}[\textbf{B}]
 Let $T\in \GL(n,\Q)$. 
\begin{itemize}
    \item Using \textbf{Algorithm II}, compute a generating set for $\C_{\GL(n,\Z)}(T) $.
    \item Transform the matrix $M$ of the bilinear form $b$, respecting the canonical basis of $\M_n(\Q)$. Call it $v$.
    \item Using \textbf{Algorithm IV}, return a generating set for the subgroup of $\C_{\GL(n,\Z)}(T)$ that verifies $gv=v$.
\end{itemize}
 
\end{algorithm}


\bibliographystyle{plain}
			\bibliography{Bib.bib}	

\sc{ \small Adem Zeghib, 
Institut Fourier, Laboratoire de mathématique, Université Grenoble Alpes, Grenoble, France.\\
Email address: adem.zeghib (at) univ-grenoble-alpes.fr} 

\end{document}